\documentclass[10pt,twoside]{article}
 \usepackage{amssymb,latexsym,amsmath,amsthm}
 
 \usepackage{geometry}
\newtheorem{thm}{Theorem}

\newtheorem{lemma}{Lemma}

\newtheorem{cor}{Corollary}

\newtheorem{con}{Conjecture}

\begin{document}

\def\L{ {\mathcal{L}}}
\def\d{ \partial }
\def\Na{{\mathbb{N}}}

\def\Z{{\mathbb{Z}}} 

\def\IR{{\mathbb{R}}}

\def\CC{{ C_c^\infty(\Omega)}}

\def\XX{{  \left\{ \psi \in C^2(\overline{\Omega}) :  \psi =0 \quad \pOm \right\} }}

\def\pOm{ \partial \Omega}

\def\HA{ (-\Delta)^\frac{1}{2}}

\def\HAN{ (-\Delta)^{\frac{-1}{2}}}

\newcommand{\E}[0]{ \varepsilon}

\newcommand{\la}[0]{ \lambda}

\newcommand{\s}[0]{ \mathcal{S}}

\newcommand{\AO}[1]{\| #1 \| }

\newcommand{\BO}[2]{ \left( #1 , #2 \right) }

\newcommand{\CO}[2]{ \left\langle #1 , #2 \right\rangle}

\newcommand{\co}[0]{ \partial B_R} 

\newcommand{\m}[0]{ B_{R} \backslash B_{\frac{R}{2}}} 

\newcommand{\p}[0]{ B_{2R} \backslash B_{\frac{R}{4}}}

  \author{Craig Cowan\footnote{Funded by NSERC.}}

\title{A Liouville theorem for a fourth order H\'enon equation} 
\maketitle

{\footnotesize
% please put the address of the first author
 \centerline{Department of Mathematics, Stanford University}
 \centerline{Building 380, Stanford, California}
 \centerline{650-723-2221}
 \centerline{ctcowan@stanford.edu}
} % Do not forget to end the {\footnotesize by the sign }

\begin{abstract}  We examine the following fourth order H\'enon  equation 
\begin{equation} \label{pipe}
\Delta^2 u = |x|^\alpha u^p  \qquad \text{in}\    \IR^N, 
\end{equation} where $ 0 < \alpha$.         Define the Hardy-Sobolev exponent  $ p_4(\alpha ):= \frac{N+4 + 2 \alpha}{N-4}$.  We show that in dimension $ N=5$ there are no positive bounded classical solutions of (\ref{pipe}) provided $ 1 < p < p_4(\alpha)$. 

 \end{abstract}

\section{Introduction and main results} 

In this, the preliminary version of the paper,  we are interested in the following fourth order problem
\begin{equation} \label{hen}
 \Delta^2 u = |x|^\alpha u^p \qquad \text{in}\    \IR^N,
 \end{equation}    where $ p>1$ and $ \alpha >0$.   
Our interest  is in the Liouville property  (ie.  the nonexistence of positive solutions).  We begin by recalling the known results for the second order analog of (\ref{hen}), 
\begin{equation} \label{second} 
-\Delta u = |x|^\alpha u^p \qquad \text{in}\    \IR^N.  
\end{equation}  
 The case where $ \alpha=0$ has been very widely studied, see \cite{Caf}, \cite{chen}, \cite{gidas},\cite{Gidas}. It is known that there are no positive classical  solutions of (\ref{second}) under various decay assumptions as $ |x| \rightarrow \infty$  provided that   
\[ 1<  p< \frac{N+2}{N-2},\] and in the case of $ N =2$ there is no solution for any $ p>1$.   We further remark that this is an optimal result.

  We now give a brief background on the case where $ \alpha $ is nonzero,   for more details see \cite{phan}. 
 We define the Hardy-Sobolev exponent 
\[ p_2(\alpha ):= \frac{N+2 + 2 \alpha}{N-2}, \]  where $  N \ge 3$.    It is known, see \cite{gidas},  that if $ \alpha <-2$ then there are no positive solutions to (\ref{second})  on any domain containing the origin.  Hence we restrict our attention to the case where $ \alpha >-2$.      
The case of radial solutions is completely understood, see \cite{gidas} and \cite{veron},  where they show there exists a positive  classical   radial solution  of (\ref{second}) if and only if  $p \ge p_2(\alpha)$.    
This result suggests the following: 
\begin{con} \label{pppp}  Suppose that $  \alpha >-2$. If $ 1 <p < p_2(\alpha)$ then (\ref{second}) has no classical bounded solution.
\end{con}   
Note that $ p_2(\alpha) \le \frac{N+2}{N-2}$  exactly when $ \alpha \le 0$.   Also note that the term $ |x|^\alpha$ changes monotonicity when $ \alpha$ changes sign.  For these reasons the methods available to prove this conjecture greatly depend on the sign of $ \alpha$.       
Until recently the best known results concerning (\ref{second}), apart from the radial case, were 
\begin{thm} \label{sou} Let $ \alpha >-2$ and $ p>1$.  

\begin{enumerate} \item If $ p < \min\{ p_2(0), p_2(\alpha) \}$ then there are no positive sufficiently regular solutions of (\ref{second}).  
\item  If $ p < \frac{N+\alpha}{N-2}$ then there are no positive weak supersolutions of (\ref{second}). 

\end{enumerate} 

\end{thm}  
The first part of this theorem is from \cite{gidas} and \cite{veron}.  Note that this implies that the Conjecture \ref{pppp} holds in the case of negative $ \alpha$.    The second part is from \cite{miti} and is an optimal result after considering  $u(x)= C |x|^\frac{-\alpha-2}{p-1}$ for some positive $ C$.  

We now come to the method which we will extend to (\ref{hen}).   
In \cite{souplet}  the Lane-Emden conjecture, which is related to the elliptic system 
\[ -\Delta u = v^p, \quad -\Delta v = u^q \qquad \IR^N, \] was shown to be true in dimension $ N=4$.  
 Later this method  was extended  in \cite{phan}  to show: 
\begin{thm} \label{pha} Suppose that $ N=3$ and $ \alpha >0$.  If $ 1 <p <p_2(\alpha)$  then there is no positive bounded classical solution of (\ref{second}). 
\end{thm}

 In fact they prove more. They show there is no positive classical solution that satisfies certain growth conditions.      This method does allow them to obtain new, but non optimal, non existence results in higher dimensions.   We also mention  that they give a new proof of Theorem \ref{sou}, 1.   
 
It is precisely Theorem \ref{pha} which we will extend to (\ref{hen}).   We remark that one could use the methods from \cite{phan} to prove results concerning (\ref{hen}) for negative $ \alpha$  but we choose not to do this.  
Before stating our result
we briefly recall the results for (\ref{hen}). 
The most studied case is when $ \alpha=0$.   In this case the  positive classical bounded solutions of (\ref{hen}) are classified and there are no positive bounded solutions provided  
\[ 1 < p < \frac{ N+4}{N-4},\]  see \cite{Lin}, \cite{Xu}.   For other results concerning (\ref{hen}) we direct the reader to \cite{Gaz},  \cite{shi}.     Using the methods from \cite{miti} one can show there are no positive weak supersolutions of (\ref{hen}) provided that 
\[ 1 < p < \frac{N+\alpha}{N-4},\]   see the comment after Lemma \ref{est}.

 In various applications one is not interested in any solution of (\ref{hen}) but rather one with added properties.    In \cite{fazly} it was shown that there are no positive finite Morse index solutions of (\ref{hen}) provided 
\[ 1 < p < p_4(\alpha):=\frac{N+4 + 2 \alpha}{N-4}.\]     We now come to our result. 

\begin{thm} \label{Main} Suppose that $ 0 < \alpha$.    Then there is no positive bounded classical solution of (\ref{hen}) provided $ N=5$ and 
\[ 1 < p < p_4(\alpha)=9+2 \alpha. \] 

\end{thm}

\section{Proof of Theorem \ref{Main}}

 We first introduce some notation.  For $ R >0$  we define $ B_R:=\{ x \in \IR^N: |x|<R\}$ and we let $ \partial B_R$ denote the boundary of the ball.  
We begin with a Rellich-Pohozaev argument.

 \begin{lemma} \cite{fazly} \label{poho} Suppose that $ u$ is a bounded nonnegative  solution of (\ref{hen}) with  $ N \ge 5$, $ \alpha>0$ and $1 < p < p_4(\alpha)$. Define \[ F(R):= \int_{B_R} |x|^\alpha u^p.\] 
 
  Then there exists  $ C=C(N,p,\alpha)>0$ but which is independent of $ R$ such that for all $ R \ge 1$ 
 \begin{eqnarray*}
\frac{F(R)}{C}  &\le & R^{1+\alpha} \int_{\partial B_R} u^{p+1}  + \int_{\co} |\Delta u| | \nabla u| \\
&& + \int_{\co} | \nabla (\Delta u)| u - \frac{R}{2C} \int_{\co}(\Delta u)^2 \\
&&+ R \int_{\co}| \nabla (\Delta u)| | \nabla u| +  \int_{\co} |\Delta u| | \nabla (x \cdot \nabla u)|. 
\end{eqnarray*}    
 \end{lemma}

 We now eliminate and simplify some of the terms from the above lemma.

 \begin{cor}  \label{poh} Suppose that $ u$ is a bounded nonnegative  solution of (\ref{hen}) with  $ N \ge 5$, $ \alpha>0$ and $1 < p < p_4(\alpha)$.  Then there is some $ C=C(N,p,\alpha)>0$ such that 
 \begin{eqnarray} \label{nonum}
 \frac{F(R)}{C} & \le &  R^{1+\alpha} \int_{\partial B_R} u^{p+1}  + \frac{1}{R} \int_{\partial B_R} | \nabla u|^2 \nonumber \\
 && + \int_{\partial B_R} | \nabla \Delta u | u  + R \int_{\partial B_R} | \nabla \Delta u | | \nabla u| \nonumber \\
 && + R \int_{\partial B_R} |D^2 u|^2 
 \end{eqnarray} for all $ R \ge 1$.   
 \end{cor}  
 For future reference we label the terms on the right hand side of (\ref{nonum}) as $ I_1,I_2,...., I_5$,  where $I_1$ is the first term on the right and $I_5$ is the last. 
 
 \begin{proof} One immediately obtains the desired result by using the following two inequalities and taking $ \E>0$ small, 
\begin{equation}  \label{x} \int_{\partial B_R} | \Delta u | | \nabla u| \le \E R  \int_{\partial B_R} | \Delta u|^2 + \frac{C_\E}{R}  \int_{\partial B_R}  | \nabla u|^2, 
\end{equation}
 \begin{equation} \label{y}
  \int_{\partial B_R} |\Delta u|  | \nabla (  x \cdot \nabla u(x))|  \le \E R \int_{\partial B_R} (\Delta u)^2 + \frac{C(\E)}{R} \int_{\partial B_R}  | \nabla u|^2  + C(\E) R \int_{\partial B_R} | D^2 u|^2.
  \end{equation}   (\ref{x}) is obtained from Young's inequality after inserting the factor $ \sqrt{R}$.  To show (\ref{y}) one begins with the inequality 
 \[ | \nabla (  x \cdot \nabla u(x))|^2 \le C \left\{ | \nabla u(x)|^2 + |x|^2 |D^2 u(x)|^2 \right\},\] where $ C= C(N)>0$.  This is obtained by direct calculation and some obvious estimates.    Then one has  
 \begin{eqnarray*}
  \int_{\partial B_R} |\Delta u|  | \nabla (  x \cdot \nabla u(x))| & \le & \E R \int_{\partial B_R} (\Delta u)^2 \\
  && + \frac{C(\E)}{R} \int_{\partial B_R} | \nabla (  x \cdot \nabla u(x))|^2 \\
  & \le & \E R \int_{\partial B_R} (\Delta u)^2 + \frac{C(\E)}{R} \int_{\partial B_R}  | \nabla u|^2 \\
  && + C(\E) R \int_{\partial B_R} | D^2 u|^2.
  \end{eqnarray*} 
  
 \end{proof}   To show the nonexistence of a positive solution of (\ref{hen}) one would like to    show that each of these boundary integral terms goes to zero as  $ R \rightarrow \infty$.      The trick in Souplet's method, see \cite{phan} and \cite{souplet},  is to view $ R$ as a parameter and then to use various Sobolev inequalities on the sphere $ S^{N-1}$  to estimate these boundary integrals.  One of the benefits of this is that one has improved embeddings since they are now working on a lower dimensional object.

 The following lemma is just the standard $L^p$ regularity and interpolation results written in a scale invariant way. 
  
 \begin{lemma}  
 
\begin{enumerate} \item ($L^p$ regularity) For $ 1 < t < \infty$ there exists $C>0$ such that for all $ R \ge 1$ and for all sufficiently regular $v$   we have 
\begin{equation} \label{reg}
 \int_{\m}   |D_x^4 v|^{t} \le C \int_{\p} |\Delta^2 v|^{t} + \frac{C}{R^{4 t}} \int_{\p} |v|^{t}. 
 \end{equation}

\item (Interpolation) For $ 1 \le t  < \infty$ and $ 1 \le j \le 3$ (an integer) there exists $ C>0$ such that for all $ R \ge 1$ and for all sufficiently regular $v$ we have 
\begin{equation} \label{interp}
 \int_{\m} |D_x^j v |^t \le C R^{t(4-j)} \int_{\p} |\Delta^2 v|^t+ \frac{C}{R^{jt}} \int_{\p} |v|^t.
 \end{equation}   
 
\end{enumerate} 

\end{lemma} 
 
The next lemma follows from the rescaled test function method from \cite{miti}

 \begin{lemma} \cite{miti} \label{est}  Suppose that $ u$ is a positive weak supersolution of  (\ref{hen}).  Then there exists  $ C>0$ such that  \[ \int_{B_{2R}} |x|^\alpha u^p \le C R^{N-\frac{\alpha}{p-1}-\frac{4p}{p-1}}\] for all $ R \ge 1$. 
 \end{lemma} 
 
 Note that if $ p < \frac{N+\alpha}{N-4}$ then the exponent of $ R$ on the right hand side  is negative and this is enough to show, after sending $ R \rightarrow \infty$, that there is no positive solution (or positive weak supersolution) to (\ref{hen}) as mentioned in the introduction.

\begin{cor}  \label{main_est} Suppose that $ N=5$, $ 1 < p < p_4(\alpha)=9+ 2 \alpha$ and that $u$ is a bounded positive classical solution of (\ref{hen}).   

 \begin{enumerate}  
\item Then there exists  $ C>0$ such that
 \[ \int_{B_{2R}} u \le C R^{N- \frac{(4+\alpha)}{p-1}}, \] for all $ R \ge 1$.

\item  Suppose that $ \E>0$. Then there exist some $C_\E>0$ such that 
\[\int_{\p} |D_x^4 u|^{1+\E} \le C_\E R^{N- \frac{4p}{p-1} - \frac{\alpha}{p-1} +\alpha \E}, \] for all $ R \ge 1$.
  
\item   There  exists  $C>0$ such that
\[ \int_{\p} |D_x^3 u| \le C R^{N - \frac{1+\alpha +3p}{p-1} }, \]  for all $ R \ge 1$.

\item There exists $ C>0$ such that 
\[ \int_{\p} |D_x u| \le C R^{ N - \frac{p+3+\alpha}{p-1}},\]  for all $ R \ge 1$.

\item There exists $ C>0$ such that 
\[  \int_{\p} |D^2_x u| \le C R^{N - \frac{\alpha+2p +2}{p-1}},\] for all $ R \ge 1$. 

 \end{enumerate}

\end{cor}

 \begin{proof} 
1. Here we use the result from Lemma \ref{est}  and H\"older's inequality. \\
2. We use the $L^p$ regularity result to see 
\begin{eqnarray*}
\int_{\m} |D_x^4 u|^{1+\E} & \le & C \int_{\p} (\Delta^2 u)^\E (\Delta^2 u) + \frac{C}{R^{1+\E}} \int_{\p} u^{1+\E}\\
& \le & C R^\E \int_{\p} |x|^\alpha u^p + \frac{C}{R^{1+\E}} \int_{\p} u
\end{eqnarray*} and one now uses Lemma \ref{est} and 1 to obtain the desired result.  \\ For the remainder of the proofs one uses the estimates from 1 and 2 and uses the interpolation inequality.

 \end{proof}

 We now introduce the various notations we will be using on the sphere.  Given nonzero $ x \in \IR^N$ we will use  spherical coordinates $ r=|x|$ and $\theta= \frac{x}{|x|} \in S^{N-1}$.   We will write $ v(x) = v(r,\theta)$.  Also unless otherwise stated $L^p$ norms will be over the unit sphere $ S^{N-1}$,  so given some function $ v $ defined on $ S^{N-1}$ we have   $ \| v\|_p:=\left\{ \int_{S^{N-1}} |v|^p d \theta \right\}^\frac{1}{p}$.  We will let  
 \[ \|v(R)\|_{p}^p:= \int_{S^{N-1}} v(R,\theta) d \theta.\]

Another key idea from \cite{souplet} is to turn  the volume estimates from Lemma \ref{est} and Corrolary \ref{main_est} 
into estimates which are valid over spheres of increasing radii $ R_m \rightarrow \infty$.  To illustrate the idea we assume that we have  estimates of the form 
\[ \int_{\m} f_i(x) dx \le C_i R^{N-\alpha_i}, \] for $ i=1,2,...,n$ and where $ 0 \le f_i$. 
The goal is to find a sequence $ R_m \rightarrow \infty$ such that 
\[ g_i(R_m):= \int_{S^{N-1}} f(R_m,\theta) d \theta,\] satisfies 
\[ g_i(R_m) \le C R_m^{-\alpha_i},\] for all $ 1 \le i \le n$.    So note that we can write 
\[ \int_{\frac{R}{2}}^R r^{N-1}g_i(r) dr = \int_{\m} f(x) dx \le C_i R^{N-\alpha_i}.\]  We now define the sets 

\[ \Gamma_i(R):=\left\{ r \in ( \frac{R}{2},R):g_i(r) \ge K R^{-\alpha_i} \right\},\] where we will pick $K>0$ later, 
and note that we have 

\[ \int_{\Gamma_i(R)} r^{N-1} g_i(r) dr \le \int_{\m} f_i(x) dx \le C_i R^{N-\alpha_i} \] and the left side has a lower bound given by 
\[ | \Gamma_i(R)| R^{N-\alpha_i-1} \frac{k}{R 2^{N-1}}.\]  This shows that 
\[ | \Gamma_i(R)| \le \frac{C_i 2^{N-1}}{K} R,\] and by taking $ K>0$ big one has, by looking at the measure, that there exists some $ \tilde{R}$ such that  
\[ \tilde{R} \in \left( \frac{R}{2}, R \right) \backslash \cup_{i=1^n} \Gamma_i(R).\]    From this one can conclude that
\[ g_i( \tilde{R}) \le  K (2^{\alpha_i} + 2^{-\alpha_i}) \tilde{R}^{-\alpha_i},\]  and so there is some $C>0$ and $ 1 \le R_m \rightarrow \infty$ such that $ g_i(R_m) \le C R_m^{-\alpha_i}$ for all $ 1 \le i \le n$.

 The following lemma is immediate after using the above procedure along with the estimates from  Corollary  \ref{main_est}.

\begin{lemma} \label{decay}  Let $ \E>0$ be small.  Then there exists some $ K>0$ and $ 1 \le R_m \rightarrow \infty$ such that 
\begin{equation}
\| D^3_x u(R_m)\|_1 \le K R_m^\frac{-\alpha-3 p-1}{p-1},
\end{equation}   
\begin{equation}
\| D_x u(R_m)\|_1 \le K R_m^\frac{-3 - \alpha-p}{p-1},
\end{equation}  
\begin{equation}
\| u(R_m)\|_1 \le K R_m^\frac{-4-\alpha}{p-1},
\end{equation} 
\begin{equation} \label{fourth}
 \| D^4_x u(R_m) \|_{1+\E}^{1+\E} \le K R_m^{\frac{-4p-\alpha}{p-1}+\alpha \E}, 
 \end{equation}
\begin{equation}
\| D_x^2 u(R_m)\|_1 \le  K R_m^\frac{-\alpha-2p-2}{p-1}.  
\end{equation}

\end{lemma}

Note that we can rewrite (\ref{fourth}) as 
\[  \| D_x^4 u(R_m)\|_{1+\E} \le  K R_m^{ \frac{-4p - \alpha}{p-1} + a(\E)},\] where
\[ a(\E)= \left( \frac{4p +\alpha}{p-1} + \alpha \right) \frac{\E}{1+\E}.\]   
Note that $ a_1(\E)>0$ all $ \E>0$ and $ a_1(\E) \rightarrow 0$ as $ \E \rightarrow 0$.

\textbf{Completion of the proof for  Theorem \ref{Main}:}  Let $ \alpha >0$, $ N =5$  and take $ 1 <p < p_4(\alpha) = 9 + 2 \alpha$.  We suppose there exists a positive bounded solution $ u$ of (\ref{hen}) and we choose $ 0 < \E < \frac{1}{100}$ small (we pick precisely later).  
  Let $ 1 \le R_m \rightarrow \infty$ be as promised in Lemma \ref{decay}.  We now find upper bounds for each term $ I_i$ from Corrolary \ref{poh} using a combination of Sobolev embeddings on the unit sphere, H\"older's inequality and the decay estimates from Lemma \ref{decay}.   We omit the index $m$ in what follows,  so $ R=R_m$.  We also omit any constants that are independent of $ R \ge 1$.  All spaces and norms are over the unit sphere.

  \begin{enumerate}

 \item (Estimate for $I_1$).  First note that $ (I_1(R))^\frac{1}{p+1} \le R^\frac{N+\alpha}{p+1} \| u(R)\|_{p+1} \le  R^\frac{N+\alpha}{p+1} \| u(R)\|_{\infty}$.    We have the Sobolev embedding $ W^{4,1+\E} \rightarrow L^\infty$ and so for any sufficiently regular $ v$ defined on the unit sphere we have  $ \| v\|_\infty \le \| D_\theta^4 v\|_{1+\E} + \| v\|_1$.  Taking $ v=u(R)$ gives 
  \[ \| u(R)\|_{\infty} \le \| D_\theta^4 u(R)\|_{1+\E} + \| u(R)\|_1 \le R^4 \| D_x u(R)\|_{1+\E} + \| u(R)\|_1,\] 
  and we now use the decay estimates from Lemma \ref{decay} to see that 
  \[ (I_1(R))^\frac{1}{p+1}   \le R^\frac{N+\alpha}{p+1} \left( R^{ \frac{-4-\alpha}{p-1} + a(\E)} + R^\frac{-4-\alpha}{p-1} \right). \]   Now recalling that $ R >1$ and that $ a(\E)>0$ we see that $ (I_1(R))^\frac{1}{p+1} \le  R^{a_1(\E)}$,  where 
  \[ a_1(\E):= \frac{p-2 \alpha -9}{(p+1)(p-1)} + a(\E), \] and we note this can be made negative by taking $ \E>0$ sufficiently small.

  \item  (Estimate for $I_2$). Firstly we have $ (I_2(R))^\frac{1}{2} \le R^\frac{N-2}{2} \| D_x u(R)\|_2$ and we have the Sobolev embedding $ W^{2,1} \rightarrow L^2$.  Hence  $ \| v \|_2 \le \|D_\theta^2 v\|_2 + \| v\|_1$ and taking $ v=D_x u(R)$ gives $ \| D_x u(R)\|_2 \le R^2 \| D^3_x u(R)\|_1 + \| D_x u(R)\|_1$.   Using the estimates from Lemma \ref{decay} gives 
  \[   (I_2(R))^\frac{1}{2} \le R^{ \frac{3}{2} - \frac{ \alpha + p + 3}{p-1}},\] and note exponent on $ R$ is negative precisely when $ p < 9 + 2 \alpha$.  

\item (Estimate for $I_3$).  First note that $ I_3(R) \le R^{N-1} \| u(R)\|_\infty \| D_x^3 u(R)\|_1$.  Using the estimates from Lemma \ref{decay} and the $L^\infty$ estimate from 1 we arrive at $ I_3(R) \le R^{a_2(\E)}$ where 
\[ a_2(\E):= \frac{p- 2 \alpha -9}{p-1} + a(\E),\]  which we note can be made negative by taking $ \E>0$ sufficiently small.   

\item  (Estimate for $I_5$).  Note that $ (I_5(R))^\frac{1}{2} \le R^\frac{N}{2} \| D_x^2 u(R)\|_2$.  Also note that $ W^{2,1} \rightarrow L^2$ so 
\[ \| D_x^2 u(R)\|_2 \le R^2 \| D_x^4 u(R)\|_1 + \|D_x^2 u(R)\|_1,\] and replacing  the term $\| D_x^4 u(R)\|_1 $ with $ \|D_x^4 u(R)\|_{1+\E}$ one arrives at 
\[  (I_5(R))^\frac{1}{2} \le R^\frac{N}{2} R^{ \frac{-2p-2-\alpha}{p-1} + a(\E)}  = R^{a_2(\E)},\] where $ a_2(\E)$ is defined above. 

\item (Estimate for $I_4$).   First note that by H\"older's inequality we have  $ I_4(R) \le R^N \| D_x^3 u(R)\|_{\frac{4}{3}} \| D_x u(R)\|_{4}$.    Sobolev embedding gives $ W^{1,1} \rightarrow L^\frac{4}{3}$ and so 
\[ \| D_x^3 u(R)\|_\frac{4}{3} \le \| D_\theta D_x^3 u(R)\|_1 + \| D_x^3 u(R)\|_1 \le R \|  D_x^4 u(R)\|_{1+\E} + \| D_x^3 u(R)\|_1.\]   Also we have $ W^{2,\frac{4}{3}} \rightarrow L^4$ and so 
\[ \| D_x u(R)\|_4 \le R^2 \| D_x^3 u(R)\|_\frac{4}{3} + \| D_x u(R)\|_1.\]   Using the estimates from Lemma \ref{decay} gives 
\[ \| D_x^3 u(R)\|_\frac{4}{3} \le R^{ \frac{-3p - \alpha -1}{p-1} + a(\E )}, \quad \| D_x u(R)\|_4 \le R^{ \frac{-p-\alpha-3}{p-1} + a(\E )}.\]  Putting this together gives $ I_4(R) \le R^{a_3(\E)}$ where  
\[ a_3(\E):= \frac{p- 9 - 2 \alpha}{p-1} + 2 a(\E), \] and note this is negative for sufficiently small $ \E>0$.

  \end{enumerate} 

So by taking $ \E>0$ but sufficiently small such that $ a_1(\E),a_2(\E),a_3(\E)<0$ and letting $ R_m$ denote the sequence promised by Lemma \ref{decay} we see that 
\[  \int_{B_{R_m}} |x|^\alpha u^{p+1}  \le C \sum_{i=1}^5 I_i(R_m), \] but $ I_i(R_m) \rightarrow 0$ as $ m \rightarrow \infty$ and so we see that $ \int_{\IR^N} |x|^\alpha u^{p+1} x =0$ contradicting the fact that $ u$ is positive.

\hfill $ \Box$

 \end{document}